\documentclass[a4paper,12pt]{amsart}
\usepackage[english]{babel}
\usepackage{amsmath,amsthm,amssymb,amsfonts}
\usepackage{multicol}
\usepackage{todonotes}
 \usepackage{enumitem} 
\usepackage[bookmarks=true,hyperindex,pdftex,colorlinks,citecolor=blue]{hyperref}
\hypersetup{colorlinks=true,  linkcolor=blue, citecolor=red, urlcolor=cyan}

\usepackage[backend=biber,style=alphabetic,maxbibnames=99]{biblatex}  
\usepackage{csquotes} 
\usepackage{xcolor}
\usepackage[export]{adjustbox}
\usepackage{graphicx}
\usepackage{subcaption}

\usepackage[english]{babel}

\usepackage{graphicx}

\usepackage{tikz}
\usetikzlibrary{mindmap,backgrounds, calc}
\usetikzlibrary{matrix,chains,scopes,positioning,arrows,fit}
\usetikzlibrary{positioning, shapes, patterns}
\usetikzlibrary{automata} 
\usetikzlibrary{shapes.geometric, arrows, arrows.meta}
\usetikzlibrary{positioning,decorations.markings}

\theoremstyle{theorem}
\newcounter{maincoro}

\newtheorem*{theoremA}{Theorem A}

\newtheorem{theorem}{Theorem}[section]
\newtheorem{lemma}[theorem]{Lemma}

\newtheorem{corollary}[theorem]{Corollary}

\theoremstyle{definition}
\newcounter{maintheorem}

\theoremstyle{remark}
\newtheorem{remark}[theorem]{Remark}
\numberwithin{equation}{section}

 \makeatletter
\renewcommand{\tocsection}[3]{%
	\indentlabel{\@ifnotempty{#2}{\bfseries\ignorespaces#1 #2\quad}}\bfseries#3}
\renewcommand{\tocsubsection}[3]{%
	\indentlabel{\@ifnotempty{#2}{\ignorespaces#1 #2\quad}}#3}

\newcommand\@dotsep{4.5}
\def\@tocline#1#2#3#4#5#6#7{\relax
	\ifnum #1>\c@tocdepth 
	\else
	\par \addpenalty\@secpenalty\addvspace{#2}%
	\begingroup \hyphenpenalty\@M
	\@ifempty{#4}{%
		\@tempdima\csname r@tocindent\number#1\endcsname\relax
	}{%
		\@tempdima#4\relax
	}%
	\parindent\z@ \leftskip#3\relax \advance\leftskip\@tempdima\relax
	\rightskip\@pnumwidth plus1em \parfillskip-\@pnumwidth
	#5\leavevmode\hskip-\@tempdima{#6}\nobreak
	\leaders\hbox{$\m@th\mkern \@dotsep mu\hbox{.}\mkern \@dotsep mu$}\hfill
	\nobreak
	\hbox to\@pnumwidth{\@tocpagenum{\ifnum#1=1\bfseries\fi#7}}\par
	\nobreak
	\endgroup
	\fi}
\AtBeginDocument{%
	\expandafter\renewcommand\csname r@tocindent0\endcsname{0pt}
}
\def\l@subsection{\@tocline{2}{0pt}{2.5pc}{5pc}{}}
\makeatother

\newcommand{\nn}[1]{{\left\vert\kern-0.25ex\left\vert\kern-0.25ex\left\vert #1 
		\right\vert\kern-0.25ex\right\vert\kern-0.25ex\right\vert}}

\renewcommand{\geq}{\geqslant}
\renewcommand{\leq}{\leqslant}

\newcommand{\vertiii}[1]{{\left\vert\kern-0.25ex\left\vert\kern-0.25ex\left\vert #1 
    \right\vert\kern-0.25ex\right\vert\kern-0.25ex\right\vert}}

\thanks{}

\subjclass[2020]{}

\date{\today}


\begin{document}

\title[Non-uniformly continuous nearest point maps]{Non-uniformly continuous nearest point maps}

\author[R. Medina]{Rubén Medina}
\address[R. Medina]{Universidad de Granada, Facultad de Ciencias. Departamento de Análisis Matemático, 18071-Granada (Spain); and Czech Technical University in Prague, Faculty of Electrical Engineering. Department of Mathematics, Technická 2, 166 27 Praha 6 (Czech Republic) \newline
\href{https://orcid.org/0000-0002-4925-0057}{ORCID: \texttt{0000-0002-4925-0057}}}
\email{rubenmedina@ugr.es}

\author[A. Quilis]{Andrés Quilis}
\address[A. Quilis]{{Universit\'{e} Franche-Comt\'{e}, Laboratoire de math\'{e}matiques de Besançon, UMR CNRS 6623, 16, route de Gray, 25000 Besançon (France)}\newline
\href{https://orcid.org/0000-0001-6022-9286}{ORCID: \texttt{0000-0001-6022-9286}}}
\email{andresqsa@gmail.com}

\thanks{}


\begin{abstract} 
We construct a Banach space satisfying that the nearest point map (also called proximity mapping or metric projection) onto any compact and convex subset is continuous but not uniformly continuous. The space we construct is locally uniformly convex, which ensures the continuity of all these nearest point maps. Moreover, we prove that every infinite-dimensional separable Banach space is arbitrarily close (in the Banach-Mazur distance) to one satisfying the above conditions.
\end{abstract}
\maketitle
\tableofcontents

\section{Introduction}
\subsection{Main result and background}

This note concerns nearest point maps (also called proximity mappings or metric projections) in Banach spaces onto compact subsets. Given a metric space $(M,d_M)$ and a subset $C\subset M$, a map $R\colon (M,d_M)\rightarrow (C,d_M)$ is called a \emph{nearest point map} if $d_M(x,R(x))=d_M(x,C)=\inf\{d_M(x,c)\colon c\in C\}$. Note that compactness of the target is sufficient, but not always necessary, to guarantee the existence of such retractions. 

Given a map $F\colon (M,d_M)\rightarrow (N,d_N)$ between metric spaces, the modulus of continuity of $F$ is a non-decreasing function $\omega_F\colon \mathbb{R}^+\rightarrow \mathbb{R}^+\cup\{+\infty\}$ given by $\omega_F(t)=\sup\{d_N(F(x),F(y))\colon x,y\in M,~ d_M(x,y)\leq t\}$. The map $F$ is uniformly continuous if and only if $\lim_{t\rightarrow 0}\omega_F(t)=0$.

Our main result is the following theorem:
\begin{theoremA}
\label{THEOREM_A}
    There exists a locally uniformly convex separable Banach space $X$ such that the nearest point map in $X$ onto any compact and convex subset is continuous but not uniformly continuous. 

    Moreover, every infinite-dimensional separable Banach space is $(1+\varepsilon)$-isomorphic to a Banach space with this property, for every $\varepsilon>0$.
\end{theoremA}

Let us put this result into context: in strictly convex Banach spaces, given a closed and convex subset $C$, if a nearest point map onto $C$ exists, then it is unique; if this is the case, we refer to it as \emph{the} nearest point map. Stronger forms of convexity of the Banach space imply additional properties of the nearest point map onto closed convex subsets. Initially, Phelps showed in \cite{Phe58} that the nearest point map from a Hilbert space $X$ onto any closed convex subset is non-expansive (1-Lipschitz). In fact, this property characterizes Hilbert spaces $X$ whenever $\text{dim}(X) \geq 3$. Later, more generally, Björnestal showed in \cite{Bjo79} that in uniformly convex Banach spaces, the nearest point map onto convex and closed subsets is uniformly continuous. It is well known that only superreflexive Banach spaces admit an equivalent uniformly convex norm.

The situation is much less restrictive when the range is assumed to be compact in addition to convex. For compact and convex subsets in strictly convex Banach spaces, the nearest point map exists, and it is unique and continuous. Brown \cite{Bro74} and Vesel\'y \cite{Ves91} showed that compactness is necessary for continuity in this setting, by constructing reflexive and strictly convex Banach spaces containing a closed (non-compact) convex subset with non-continuous nearest point map. Moreover, in the compact framework, in order to obtain uniform continuity of the nearest point map, we do not need the full strength of uniform convexity of the Banach space. Indeed, Hájek and Medina showed in \cite[Proposition 6.7]{HM21} that given a compact $K$ in a Banach space $X$, if $X$ is uniformly convex in the direction of $x$ for every $x\in S_X \cap\overline{\text{span}(K)}$, then the nearest point map onto $K$ is uniformly continuous. This means, in particular, that every separable Banach space is $(1+\varepsilon)$-isomorphic to a Banach space in which the nearest point map onto any compact convex set is uniformly continuous, for any $\varepsilon>0$. 

In turn, this implies that compact and convex subsets of Banach spaces are absolute uniform retracts. Moreover, Medina showed in \cite{Med23} that, for every $\alpha<1$, a rich class of compact and convex subsets of Banach spaces are actually absolute $\alpha$-Hölder retracts. The term ``rich" here refers to the fact that only certain assumptions on the asymptotic shape of the compact and convex set are needed, which are weak enough such that every separable Banach space contains a compact and convex set with this property and whose closed linear span is the whole space.

A closely related problem by Godefroy and Ozawa (asked in \cite{GO14}) is whether in every separable Banach space there is a Lipschitz retraction onto a compact and convex subset whose closed linear span in the whole space. Theorem \ref{THEOREM_A} is a strong counterexample to a version of Godefroy and Ozawa's question, where we only consider the retractions given by nearest point maps. Furthermore, our example shows that the above-mentioned result \cite[Proposition 6.7]{HM21} cannot be generalized to strictly convex Banach spaces, or even locally uniformly convex.

It is worth mentioning that compactness is crucial in our result. Indeed, in every Banach space there is a $2$-Lipschitz nearest point map onto a bounded convex subset, namely the radial projection onto the unit ball.

\subsection{Outline}

Let us intuitively explain how we arrive to our main result, while describing the content of the two main Sections \ref{blueprint} and \ref{main}. The purpose of this subsection is purely explanatory, so we will not be fully rigorous. This subsection can be skipped, as the rest of the note is self-contained and still contains some brief explanatory remarks. 

We start with an arbitrary infinite-dimensional separable Banach space $(X,\|\cdot\|)$ and a parameter $0<\rho<1/4$, and we will construct a locally uniformly convex equivalent norm $\vertiii{\cdot}_\rho$ in $X$ which satisfies Theorem \ref{THEOREM_A}, and which approximates $\|\cdot\|$ arbitrarily well as $\rho$ goes to $0$. This will clearly prove all statements of Theorem \ref{THEOREM_A}. Note that we may suppose that the starting norm $\|\cdot\|$ is locally uniformly convex, since every separable Banach space is $(1+\varepsilon)$-isomorphic to a locally uniformly convex space for every $\varepsilon>0$ (we briefly discuss this fact in Subsection \ref{Preliminaries}).

First, we construct an equivalent norm in $X$ with much weaker properties, and which depends on some set of parameters $\alpha$. This norm, denoted by $\|\cdot\|_\alpha$, is not even strictly convex, but satisfies a very concrete version of Theorem \ref{THEOREM_A}, and acts as the blueprint for an infinite family of norms we construct in this note. Namely, in $(X,\|\cdot\|_\alpha)$, there exist two different points $x^+_\alpha,x^-_\alpha\in X$ which are closer than some $t_\alpha>0$ (i.e.: $\|x^+_\alpha-x^-_\alpha\|_\alpha<t_\alpha$) and such that any nearest point map $R$ onto certain compact and convex subsets (depending on $\alpha$ as well) it holds that $R(x^+_\alpha)$ and $R(x^-_\alpha)$ are further than $C\rho$, where $C$ is a universal constant (i.e.: $\|R(x^+_\alpha)-R(x^-_\alpha)\|_\alpha>C\rho$). This gives a lower bound on the modulus of continuity of nearest point maps at $t_\alpha$ which only depends on $\rho$, but which only works for certain compact convex subsets. The norm $\|\cdot\|_\alpha$ is constructed based on three vectors in $X$, and it can be faithfully represented in a picture (see Figure \ref{fig:Top_half_unit_ball}).

The idea is that, in order to deal with all compact and convex subsets in $X$, we only need the norms associated to countably many sets of parameters $\alpha$. Similarly, the parameter $t_\alpha>0$ can be made as small as we wish while obtaining the same lower bound on the modulus of continuity; so, again, only countably many norms of the form $\|\cdot\|_\alpha$ are required to fully deny the uniform continuity of nearest point maps. It is also important the fact that, even though $\|\cdot\|_\alpha$ is not strictly convex, sufficiently good locally uniformly convex approximations will clearly satisfy the same result (with slightly different parameters). Again, only countably many approximations of a given norm $\|\cdot\|_\alpha$ will be enough to reproduce its properties. 

The purpose of Section \ref{blueprint} is to construct the norm $\|\cdot\|_\alpha$ for a suitable set of parameters $\alpha$, and to show that good enough approximations to $\|\cdot\|_\alpha$ have the desired concrete version of Theorem \ref{THEOREM_A} above discussed. 

In short, after Section \ref{blueprint} we are able to define a sequence of locally uniformly convex norms $(\|\cdot\|_n)_n$ which yield the same lower bound on the modulus of continuity of nearest point maps for some compact and convex subsets, when evaluated at increasingly smaller $t_n>0$; and which all together deal with any compact convex subset in $X$. We only need to combine all of them into a single locally uniformly convex norm $\vertiii{\cdot}_\rho$ which does all of this simultaneously. We do this in Section \ref{main}, where we finally prove that $\vertiii{\cdot}_\rho$ satisfies Theorem \ref{THEOREM_A}. The combination of $(\|\cdot\|_n)_n$ into a single norm $\vertiii{\cdot}_\rho$ is rather simple: the unit ball of $\vertiii{\cdot}_\rho$ is the intersection of the original unit ball and all unit balls of the sequence of norms $(\|\cdot\|_n)_n$. The reason this intersection works for our purposes is mainly due to two factors: 

In the first place, because of their construction, each locally uniformly convex norm $\|\cdot\|_n$ in the sequence $(\|\cdot\|_n)_n$ coincides with the original norm except in a particular slice $S_n$ (i.e.: an intersection of the unit ball with a hyperplane) and its polar opposite $-S_n$, where the new unit ball is smaller than the original. This means that intersecting the original ball with the unit ball of $\|\cdot\|_n$ only changes these particular slices $\pm S_n$. In the slice $S_n$, there is a critical region of the new unit sphere where the precise geometry is present in order to have the properties we need of $\|\cdot\|_\alpha$. The final norm $\vertiii{\cdot}_\rho$ must leave the critical region of each norm in the sequence $(\|\cdot\|_n)_n$ intact for it to satisfy Theorem \ref{THEOREM_A}. Therefore, we must choose the sequence $(\|\cdot\|_n)_n$ in a way such that the slice $S_n$ does not intersect the critical region of any other norm $\|\cdot\|_m$ if $m\neq n$. 

This leads to the second and final ingredient: for each $n\in\mathbb{N}$, the slice $S_n$ and the critical region within this slice are determined by a pair $(e_n,e^*_n)\in S_{(X,\|\cdot\|)}\times S_{(X^*,\|\cdot\|)}$ with $e^*_n(e_n)=1$. More precisely, the slice $S_n$ is of the form $\{x\in B_{(X,\|\cdot\|)}\colon e^*_n(x)>1-\rho\}$, and the critical region in $S_n$ contains the point $(1-\rho/2)e_n$ and has small diameter (roughly of the order of $\rho$). Therefore, if we choose an almost biorthogonal sequence $(e_n,e_n^*)_n\subset S_{(X,\|\cdot\|)}\times S_{(X^*,\|\cdot\|)}$, we will ensure that the critical regions of each $\|\cdot\|_n$ are not in any other slice $S_m$ for $m\neq n$. In order for $\vertiii{\cdot}_\rho$ to be equivalent and locally uniformly convex, we will also require that any given point in $S_{(X,\|\cdot\|)}$ has a neighbourhood which only intersects finitely many slices. For this reason, we further need the sequence $(e_n^*)_n$ to be weak$^*$-null. The existence of an almost biorthogonal sequence with these requirements in every infinite-dimensional Banach space is a standard consequence of Josefson$-$Nissenzweig Theorem.

\subsection{Notation and preliminaries}\label{Preliminaries}

We finish the introduction by fixing the notation and by presenting some more basic definitions and results that will be used throughout the article. 

We will consider real Banach spaces. The unit ball of a Banach space $(X,\|\cdot\|)$ is denoted by $B_{(X,\|\cdot\|)}$, and the unit sphere by $S_{(X,\|\cdot\|)}$. The dual space of $X$ is denoted by $X^*$, and the dual norm in $X^*$ associated to $\|\cdot\|$ is denoted again by $\|\cdot\|$. 

In a Banach space $X$, two norms $\|\cdot\|_1$ and $\|\cdot\|_2$ are \emph{equivalent} if $(X,\|\cdot\|_1)$ and $(X,\|\cdot\|_2)$ are isomorphic. Given $\varepsilon>0$, we say that two Banach spaces $(X,\|\cdot\|_X)$ and $(Y,\|\cdot\|_Y)$ are \emph{$(1+\varepsilon)$-isomorphic} if there exists an isomorphism $T\colon (X,\|\cdot\|_X)\rightarrow (Y,\|\cdot\|_Y)$ such that $\|T\|\|T^{-1}\|\leq (1+\varepsilon)$. Clearly, if $\|\cdot\|_1$ and $\|\cdot\|_2$ are two norms in a Banach space $X$ such that $a\|\cdot\|_1\leq \|\cdot\|_2\leq b\|\cdot\|_1$, then $(X,\|\cdot\|_1)$ and $(X,\|\cdot\|_2)$ are $b/a$-isomorphic.

We say that a Banach space $(X,\|\cdot\|)$ is \emph{strictly convex} if whenever $x,y\in S_{(X,\|\cdot\|)}$ satisfy $\frac{x+y}{2}\in S_{(X,\|\cdot\|)}$, then $x=y$. We say that $(X,\|\cdot\|)$ is \emph{locally uniformly convex} if for every $x\in S_{(X,\|\cdot\|)}$ and every sequence $(y_n)_n\subset S_{(X,\|\cdot\|)}$ such that $\left\|\frac{x+y_n}{2}\right\|\rightarrow 1$, it holds that $\|x-y_n\|\rightarrow 0$. It is straightforward to see that every locally uniformly convex Banach space is strictly convex. As mentioned in the first subsection, in a strictly convex Banach space, the nearest point map onto any compact and convex subset exists, and is unique and continuous. 

In a Banach space $(X,\|\cdot\|)$, the function $Q_{\|\cdot\|}\colon X\times X\rightarrow \mathbb{R}^+$ given by $Q_{\|\cdot\|}(x,y)=2\|x\|^2+2\|y\|^2-\|x+y\|^2$ for all $(x,y)\in X\times X$ encodes some information about convexity properties of $(X,\|\cdot\|)$. Indeed, $(X,\|\cdot\|)$ is strictly convex if and only if $Q_{\|\cdot\|}(x,y)=0$ implies that $x=y$, and $(X,\|\cdot\|)$ is locally uniformly convex if and only if for every $x\in X$ and every sequence $(y_n)_n\subset X$, if $Q_{\|\cdot\|}(x,y_n)\rightarrow 0$ then $\|x-y_n\|\rightarrow 0$. Thanks to this function, it is straightforward to check that if both $(X,\|\cdot\|_1)$ and $(X,\|\cdot\|_2)$ are locally uniformly convex, then $(X,\max\{\|\cdot\|_1,\|\cdot\|_2\})$ is also locally uniformly convex. This result clearly extends to finitely many equivalent norms in $X$. 

We will use that every separable Banach space is $(1+\varepsilon)$-isomorphic to a locally uniformly convex Banach space. This follows from Kadets' classical result, which states that every separable Banach space $(X,\|\cdot\|_1)$ admits an equivalent norm $\|\cdot\|_2$ such that $(X,\|\cdot\|_2)$ is locally uniformly convex. Then, the norm $\vertiii{\cdot}=(\|\cdot\|_1+\varepsilon\|\cdot\|_2)^{1/2}$ approximates $\|\cdot\|_1$, and $(X,\vertiii{\cdot})$ is locally uniformly convex as well, as can be easily checked thanks to the function $Q_{\vertiii{\cdot}}$. We refer the reader to the monographs \cite{DevGodZiz93} and \cite{GuiMonZiz22} for a more in depth study of these concepts.

\section{Blueprint norm}\label{blueprint}

Fix $0<\rho<1/4$ for the rest of the article, and fix an infinite-dimensional, locally uniformly convex separable Banach space $(X,\|\cdot\|)$. 





In this section we describe the construction of a single equivalent norm $\|\cdot\|_\alpha$ in $X$, depending on a tuple $\alpha=(v,v^*,e,e^*,h,h^*,t)$ with $v,e,h\in S_{(X,\|\cdot\|)}$, $v^*,e^*,h^*\in S_{(X^*,\|\cdot\|)}$ and $t>0$. We require that $t<\rho/16$ and that
\begin{equation}
\label{eq:biorthogonality_alpha}
v^*(v) =e^*(e) =h^*(h) = 1,\qquad\text{and }|e^*(h)|<\frac{\rho}{800} .
\end{equation}
For the remainder of the section, we define the following $4$ vectors in $X$:
\begin{align}
\label{defelem}
x^{\pm}_\alpha&=\Big(1-\frac{\rho}{2}\Big)e\pm t h,\\
y^+_\alpha&=x^+_\alpha+\frac{\rho}{4}v,\qquad y^-_\alpha=x^-_\alpha-\frac{\rho}{4}v.
\end{align}
Note that $\|y^+_\alpha\|,\|y^-_\alpha\|\leq 1-\frac{\rho}{2}+\frac{\rho}{16}+\frac{\rho}{4}<1-\frac{\rho}{8}$.

Let $S_\alpha$ be the slice of $B_{(X,\|\cdot\|)}$ given by $e^*$ and $\rho$, that is,
$$ S_\alpha=\{x\in B_{(X,\|\cdot\|)}:e^*(x)>1-\rho\}.$$
Now, we define the closed, convex and symmetric set
$$ B_\alpha =\text{co}\big((B_{(X,\|\cdot\|)}\setminus\pm S_\alpha)\cup \{y^\pm_\alpha,-y^\pm_\alpha\}\big)\subset B_{(X,\|\cdot\|)}.$$
Clearly, $(1-\rho)B_{(X,\|\cdot\|)}$ is contained in $B_\alpha$, so the Minkowski functional defined by $B_\alpha$ defines an equivalent norm $\|\cdot\|_\alpha$ such that 
\begin{equation}
\label{eq:Eq_Alpha_norms}
    \|\cdot\|\leq \|\cdot\|_\alpha\leq (1-\rho)^{-1}\|\cdot\|.
\end{equation}
The construction of $\|\cdot\|_\alpha$ depends on three directions and a functional, given by the vectors $v,e$ and $h$ and the functional $e^*$. Hence, we can realize it in a three dimensional space. In Figure \ref{fig:Top_half_unit_ball} we have a representation of the top half of $B_\alpha$ if we were to construct it in the three-dimensional euclidean space, using the standard biorthogonal coordinate basis.
\begin{figure}
    \centering
    \includegraphics[width=\textwidth]{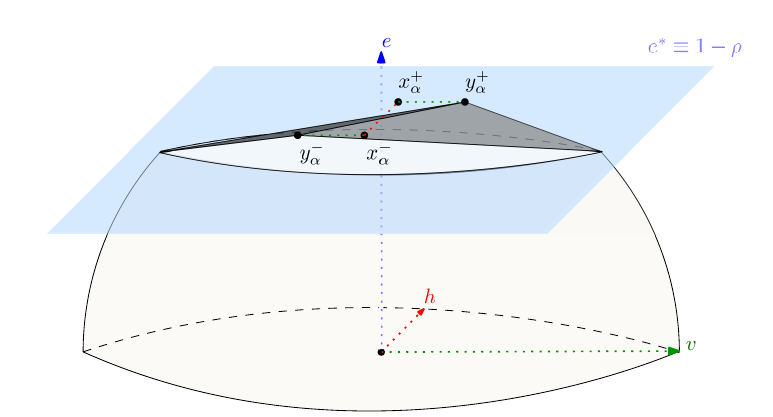}
    \caption{Top half of the unit ball of $\|\cdot\|_\alpha$ in the three-dimensional euclidean space using the vectors from the canonical biorthogonal basis.}
    \label{fig:Top_half_unit_ball}
\end{figure}

\begin{remark}
    The functionals $v^*$ and $h^*$ do not play any role in the definition of $\|\cdot\|_\alpha$, and thus they do not need to be fixed from the beginning as the rest of the elements in the tuple $\alpha$. However, as it has no effect in the construction, we choose to include them in the tuple $\alpha$ in an attempt to make the statements in this section more concise.
\end{remark}

The key feature of the norm $\|\cdot\|_\alpha$ is that vectors of the form $x^\pm_\alpha+k$ have norm strictly bigger than $1$, provided that the vectors $v$ and $k\in X$ are sufficiently orthogonal to both $e^*$ and $h^*$. In order to prove this, we will separate the unit ball $B_{(X,\|\cdot\|_\alpha)}=B_\alpha$ from such points using two particular hyperplanes $\varphi^\pm\in X^*$. Intuitively, the hyperplanes $\varphi^\pm$ are defined by slightly tilting the hyperplane $e^*\equiv 1-\frac{\rho}{2}$ with another hyperplane orthogonal to $y^+_\alpha-y^-_\alpha$. In the following technical lemma we collect, for later reference, some estimates of the image of these hyperplanes on the relevant set of vectors. Note that the estimates we obtain, while sufficient for our purposes, are not optimal, since we opted to prioritize simplicity over sharpness. 

\begin{lemma}
\label{Lemma:Estimate_image_hyperplane}
    Let $\alpha=(v,v^*,e,e^*,h,h^*,t)$ as above. Let $\lambda=1-\frac{\rho}{100}\in \left(\frac{1}{2},1\right)$, and consider: 
    \begin{align*}
        \varphi^+ &= \lambda e^*+(1-\lambda)\left(h^*-\frac{4t}{\rho}v^*\right)\in X^*\\
        \varphi^- &= \lambda e^*+(1-\lambda)\left(-h^*+\frac{4t}{\rho}v^*\right)\in X^*.
    \end{align*}
    Let $k\in B_{(X,\|\cdot\|)}$. Then, we have:

    \begin{enumerate}
        \item $\varphi^+\left(x^+_\alpha+k-y^+_\alpha\right)\geq \frac{t\rho}{25}\left(\frac{1}{16}-\frac{1}{\rho}v^*(k)\right)-\left|e^*\left(k-\frac{\rho}{4}v\right)\right|-\left|h^*\left(k-\frac{\rho}{4}v\right)\right|$.

        \item $\varphi^+\left(x^+_\alpha+k-y^-_\alpha\right)\geq \frac{t\rho}{25}\left(\frac{1}{16}-\frac{1}{\rho}v^*(k)\right)-\left|e^*\left(k+\frac{\rho}{4}v\right)\right|-\left|h^*\left(k+\frac{\rho}{4}v\right)\right|$.

        \item $\varphi^+\left(x^+_\alpha+k-z\right)\geq \frac{\rho}{8}-|e^*(k)|$, for all $z\in B_{(X,\|\cdot\|)}$ such that $e^*(z)\leq(1-\rho)$.
    \end{enumerate}

    Similarly, we also obtain:

    \begin{enumerate}
        \item $\varphi^-\left(x^-_\alpha+k-y^-_\alpha\right)\geq \frac{t\rho}{25}\left(\frac{1}{16}+\frac{1}{\rho}v^*(k)\right)-\left|e^*\left(k+\frac{\rho}{4}v\right)\right|-\left|h^*\left(k+\frac{\rho}{4}v\right)\right|$.
     
        \item $\varphi^-\left(x^-_\alpha+k-y^+_\alpha\right)\geq \frac{t\rho}{25}\left(\frac{1}{16}+\frac{1}{\rho}v^*(k)\right)-\left|e^*\left(k-\frac{\rho}{4}v\right)\right|-\left|h^*\left(k-\frac{\rho}{4}v\right)\right|$.

        \item $\varphi^-\left(x^-_\alpha+k-z\right)\geq \frac{\rho}{8}-|e^*(k)|$, for all $z\in B_{(X,\|\cdot\|)}$ such that $e^*(z)\leq(1-\rho)$.
    \end{enumerate}
\end{lemma}
\begin{proof}
The proof consists of elementary computation. We will show the result for $\varphi^+$ and $x^+_\alpha$, since the second part is shown similarly. 

Fix $k\in B_{(X,\|\cdot\|)}$. For item (1), we have:
\begin{align*}
    \varphi^+(x^+_\alpha+k-y^+_\alpha)&=\varphi^+\left(k-\frac{\rho}{4}v\right)\\
    &=\lambda e^*\left(k-\frac{\rho}{4}v\right)+(1-\lambda)h^*\left(k-\frac{\rho}{4}v\right)-(1-\lambda)\frac{4t}{\rho}v^*(k)+(1-\lambda)t\\
    &\geq \frac{t\rho}{100}\left(1-\frac{4}{\rho}v^*(k)\right)-\left|e^*\left(k-\frac{\rho}{4}v\right)\right|-\left|h^*\left(k-\frac{\rho}{4}v\right)\right|
\end{align*}
which in particular implies the estimate we need. 
We have a similar process for item (2):
\begin{align*}
    \varphi^+(x^+_\alpha+k-y^-_\alpha)&=\varphi^+\left(k+2th+\frac{\rho}{4}v\right)\\
    &=\lambda e^*\left(k+\frac{\rho}{4}v\right)+2t\lambda e^*(h)+(1-\lambda)h^*\left(k+\frac{\rho}{4}v\right)+(1-\lambda)2t\\
    &\phantom{=}+(1-\lambda)\left(-\frac{4t}{\rho}v^*(k)-\frac{8t^2}{\rho}v^*(h)-t\right)\\
    &= t\left((1-\lambda)\left(2-\frac{4}{\rho}v^*(k)-\frac{8t}{\rho}v^*(h)-1\right)+2\lambda e^*(h)\right)\\
    &\phantom{=}+\lambda e^*\left(k+\frac{\rho}{4}v\right)+(1-\lambda)h^*\left(k+\frac{\rho}{4}v\right)\\
    &\geq t\left((1-\lambda)\left(\frac{1}{2}-\frac{4}{\rho}v^*(k)\right)-2 |e^*(h)|\right)\\
    &\phantom{=}-\left|e^*\left(k+\frac{\rho}{4}v\right)\right|-\left|h^*\left(k+\frac{\rho}{4}v\right)\right|\\
    &\geq \frac{t\rho}{25}\left(\frac{1}{16}-\frac{1}{\rho}v^*(k)\right)-\left|e^*\left(k+\frac{\rho}{4}v\right)\right|-\left|h^*\left(k+\frac{\rho}{4}v\right)\right|,
\end{align*}
where we have used that $\frac{8t}{\rho}<\frac{1}{2}$ and $|e^*(h)|<\frac{\rho}{800}=\frac{(1-\lambda)}{8}$.

For claim (3), observe that the choice of $\lambda$ implies in particular that $5(1-\lambda)<\frac{\rho}{16}$. With this, for any $z\in X$ with $\|z\|\leq 1$ and $e^*(z)<1-\rho$ we obtain:

\begin{align*}
    \varphi^+(x^+_\alpha+k-z)&=\varphi^+\left(\left(1-\frac{\rho}{2}\right)e+th+k-z\right)\\
    &=\lambda\left(1-\frac{\rho}{2}-e^*(z)\right)+\lambda (te^*(h)+e^*(k))\\
    &\phantom{=}+(1-\lambda)\left(t+h^*\left(\left(1-\frac{\rho}{2}\right)e+k-z\right)\right)\\
    &\phantom{=}-(1-\lambda)\frac{4t}{\rho}v^*\left(\left(1-\frac{\rho}{2}\right)e+th+k-z\right)\\
    &\geq \lambda\frac{\rho}{2}-\lambda t|e^*(h)|-\lambda |e^*(k)|-(1-\lambda)(3+t)-(1-\lambda)\frac{1}{4}(3+t)\\
    &\geq \frac{\rho}{4}-\frac{\rho}{16}-5(1-\lambda)-|e^*(k)|\geq \frac{\rho}{8}-|e^*(k)|.
\end{align*}
In the second to last inequality, we used the (suboptimal) bound $\lambda t|e^*(h)|<\frac{\rho}{16}$ given by the condition $t<\frac{\rho}{16}$ we imposed for $\alpha$.
\end{proof}

We now state and prove the main result of this section. The core argument of its proof is the following: In the space $(X,\|\cdot\|_\alpha)$, whenever we have a convex compact set $K$ containing $\pm \frac{\rho}{4} v$, any nearest point map $R$ onto $K$ satisfies that $\|x^\pm_\alpha-R(x^\pm_\alpha)\|_\alpha\leq 1$, since the points $\mp \frac{\rho}{4} v\in K$ in place of $R\left(x^\pm_\alpha\right)$ satisfy this inequality. Therefore, if the vectors in $K$ (which include $\pm\frac{\rho}{4}v$) are sufficiently orthogonal to $e^*$ and $h^*$, using the previous lemma we deduce that $-v^*\left(R(x^+_\alpha)\right)$ and $v^*\left(R(x^-_\alpha)\right)$ must at least be strictly bigger than $\frac{\rho}{16}$, since otherwise the vectors $x^\pm_\alpha - R(x^\pm_\alpha)$ will be separated from the unit ball $B_\alpha$ by the hyperplanes $\varphi^\pm$. Therefore, $v^*$ separates the vectors $R(x^+_\alpha)$ and $R(x^-_\alpha)$ by $\frac{\rho}{8}$, which yields a lower bound on the modulus of continuity of $R$ for the parameter $(1-\rho)^{-1}2t\geq\|x^+_\alpha-x^-_\alpha\|_\alpha$. Importantly, this lower bound only depends on $\rho$. 

The statement of the theorem, and consequently its proof, are slightly more technical than the above paragraph. This is due to the fact that the final norm is constructed by approximating infinitely many norms of the form $\|\cdot\|_\alpha$, and thus we already include some approximating considerations in this section. As in the previous lemma, for the sake of simplicity and readability, the conditions we assume for the compact $K$ and the estimates we obtain are not optimal.

\begin{theorem}
\label{Th:Estimate_modulus_NPM}
    Let $\alpha=(v,v^*,e,e^*,h,h^*,t)$ as above. Let $\eta>0$ such that $\eta<\frac{t\rho}{12800}$. Suppose that $(X,\vertiii{\cdot})$ is strictly convex, and suppose that $\|\cdot\|_\alpha\leq\vertiii{\cdot}\leq (1-\rho)^{-2}\|\cdot\|$ and $\vertiii{y^\pm_\alpha}\leq (1+\eta)$. 

    Let $K$ be a convex and compact set in $B_{(X,\vertiii{\cdot})}$ such that there exists $v_K\in X$ with $\frac{\rho}{4}\vertiii{v-v_K}<\eta$ and $\left[-\frac{\rho}{4} v_K,\frac{\rho}{4} v_K\right]\subset K$. If $\sup_{k\in K}|e^*(k)|,\sup_{k\in K}|h^*(k)|<\frac{t\rho}{12800}$, then the nearest point map $R\colon (X,\vertiii{\cdot})\rightarrow (K,\vertiii{\cdot})$ satisfies:
    $$\omega_{R}\left((1-\rho)^{-2}2t\right)\geq \frac{\rho}{16}.$$
\end{theorem}
\begin{proof}
     We will show that $\vertiii{R(x^+_\alpha)-R(x^-_\alpha)}\geq \frac{\rho}{16}$, which implies the desired conclusion. We will do it by proving that $v^*\left(R(x^+_\alpha)\right)\leq-\frac{\rho}{32}$ and $v^*\left( R(x^-_\alpha)\right)\geq\frac{\rho}{32}$. Since $\|v^*\|=1$ and $\vertiii{\cdot}$ is a bigger norm, the result follows.

    Notice that, since $\left|e^*\left(\pm\frac{\rho}{4}v_K\right)\right|,\left|h^*\left(\pm\frac{\rho}{4}v_K\right)\right|<\frac{t\rho}{12800}$, it follows that $\left|e^*\left(\pm\frac{\rho}{4}v\right)\right|,\left|h^*\left(\pm\frac{\rho}{4}v\right)\right|<\frac{t\rho}{6400}$. 
    
    We prove the claim for $R(x^+_\alpha)$; the proof for $R(x^-_\alpha)$ is analogous. Suppose by contradiction that $v^*\left(R(x^+_\alpha)\right)>-\frac{\rho}{32}$. Observe first that 
    \begin{align*}
        \vertiii{x^+_\alpha-\left(-\frac{\rho}{4} v_K\right)}\leq \vertiii{y^+_\alpha}+\frac{\rho}{4}\vertiii{v-v_K}\leq 1+2\eta.
    \end{align*}
    Since $-\frac{\rho}{4} v_K$ belongs to $K$, this implies that $\vertiii{x^+_\alpha-R(x^+_\alpha)}$ is at most $1+2\eta$. Consider the functional $ \varphi^+ = \lambda e^*+(1-\lambda)\left(h^*-\frac{4t}{\rho}v^*\right)\in X^*$ with $\lambda=\frac{\rho}{100}$. To arrive at a contradiction, it suffices to show that, for any point $u$ in $(1+2\eta)B_{(X,\vertiii{\cdot})}$ we have $\varphi^+\left(x^+_\alpha-R(x^+_\alpha)-u\right)>0$.  
    
    Using Lemma \ref{Lemma:Estimate_image_hyperplane} with $k = -R(x^+_\alpha)$, we have that 
    \begin{align*}
        \varphi^+\left(x^+_\alpha-R(x^+_\alpha)-y^+_\alpha\right)&> \frac{t\rho}{800}-2\frac{t\rho}{6400}-2\frac{t\rho}{12800}>\frac{t\rho}{1600},\\
        \varphi^+\left(x^+_\alpha-R(x^+_\alpha)-y^-_\alpha\right)&> \frac{t\rho}{800}-2\frac{t\rho}{6400}-2\frac{t\rho}{12800}>\frac{t\rho}{1600}, \\
        \varphi^+\left(x^+_\alpha-R(x^+_\alpha)-z\right)&> \frac{\rho}{8}-\frac{t\rho}{12800}>\frac{t\rho}{1600},\quad\text{for all }z\in B_{(X,\|\cdot\|)}\text{ with }e^*(z)<1-\rho.
    \end{align*} 
    A simple computation shows that $\|\varphi^+\|\leq 1+\frac{t}{25}$. Therefore, we have that 
    \begin{align*}
        \varphi^+\left(x^+_\alpha-R(x^+_\alpha)-(1+2\eta)y^+_\alpha\right)&> \frac{t\rho}{1600}-\left(1+\frac{t}{25}\right)\frac{t\rho}{6400}>0,\\
        \varphi^+\left(x^+_\alpha-R(x^+_\alpha)-(1+2\eta)y^-_\alpha\right)&> \frac{t\rho}{1600}-\left(1+\frac{t}{25}\right)\frac{t\rho}{6400}>0, \\
        \varphi^+\left(x^+_\alpha-R(x^+_\alpha)-(1+2\eta)z\right)&> \frac{t\rho}{1600}-\left(1+\frac{t}{25}\right)\frac{t\rho}{6400}>0,
    \end{align*} 
    for all $z\in B_{(X,\|\cdot\|)}$ with $e^*(z)<1-\rho$.
    Since $B_{(X,\vertiii{\cdot})}$ is contained in $B_\alpha$, every point in $(1+2\eta)B_{(X,\vertiii{\cdot})}$ can be written as a convex combination of $(1+2\eta)y^+_\alpha$, $(1+2\eta)y^-_\alpha$ and $(1+2\eta)z$ with $z\in B_{(X,\|\cdot\|)}$ and $e^*(z)<1-\rho$. This leads to the contradiction we sought. 
\end{proof}

\section{Proof of main theorem}\label{main}

Once we have constructed the blueprint for the norms we will be considering, we proceed to obtain locally uniformly convex approximations of norms of the form $\|\cdot\|_\alpha$, and to combine these approximations into one final equivalent norm which satisfies that the nearest point map onto any non-trivial convex compact set is continuous but fails to be uniformly continuous. 

First, we define a suitable countable set of tuples $(\alpha_n)_n=(v_n,v_n^*,e_n,e_n^*,h_n,h^*_n,t_n)_n$ that will yield the norms to approximate and combine. Let $(A_i)_i$ be an infinite partition of $\mathbb{N}$ formed by infinite sets. Define, for every $n\in\mathbb{N}$, the positive number $t_n=2^{-i}\frac{\rho}{16}$, where $i$ is the unique natural number such that $n\in A_i$. As seen in Theorem \ref{Th:Estimate_modulus_NPM}, the number $t_n$ determines for which parameter of the modulus of continuity we obtain the lower bound, and so we need the sequence $(t_n)_n$ to be arbitrarily close to $0$. 

Since the final construction works for every non-trivial convex compact set, we fix a sequence of vectors $(v_n)_n\subset S_{(X,\|\cdot\|)}$ such that, for every $i\in\mathbb{N}$, the set $\{v_n\colon n\in A_i\}$ is dense in $S_{(X,\|\cdot\|)}$ (we may do this by simply repeating the same dense sequence in each $A_i$). Define for each $n\in\mathbb{N}$ a functional $v_n^*\in S_{(X^*,\|\cdot\|)}$ such that $v_n^*(v_n)=1$. 

To finish defining $\alpha_n$ for every $n\in\mathbb{N}$, it only remains to define the sequences $(e_n)_n,(h_n)_n$ in $S_{(X,\|\cdot\|)}$ and the corresponding norming functionals $(e^*_n)_n,(h^*_n)_n$ in $S_{(X^*,\|\cdot\|)}$. We will in fact consider just two suitable sequences $(e_n)_n$ and $(e_n^*)_n$, and put $h_n=e_{n+1}$ and $h^*_n=e^*_{n+1}$. Note that, in this way, the system $(e_n,e^*_n)_n$ not only determines the slice we modify to construct $\|\cdot\|_{\alpha_n}$, but also where the critical points $x^\pm_{\alpha_n}$ and $y^\pm_{\alpha_n}$ lie within this slice. Therefore, if we hope to preserve the precise geometry of the norms $\|\cdot\|_{\alpha_n}$, we must choose the system $(e_n,e^*_n)_n$ in a way that the slices modified in each $\|\cdot\|_{\alpha_n}$ do not intersect the critical region of every other one. This will be done by considering an almost biorthogonal system: we choose two sequences $(e_n)_n\subset S_{(X,\|\cdot\|)}$ and $(e^*_n)_n\subset S_{(X^*,\|\cdot\|)}$ such that $(e^*_n)_n$ converges weak$^*$ to $0$, and such that $e^*_n(e_n)=1$ for all $n\in\mathbb{N}$ and $|e^*_n(e_m)|<\frac{\rho}{800}$ for all $n\neq m\in \mathbb{N}$. The existence of such an almost biorthogonal system in every infinite-dimensional Banach space is a standard fact that can be deduced from Josefson$-$Nissenzweig's theorem (see e.g.: Claim 1 in \cite{Qui23} for a proof). Note that the almost biorthogonality also guarantees that $|e^*_n(e_{n+1})|<\frac{\rho}{800}$, as required in equation \eqref{eq:biorthogonality_alpha} to define $\|\cdot\|_{\alpha_n}$. The fact that $(e^*_n)_n$ is weak$^*$ null is also crucial, since it is necessary to show that the final space is locally uniformly convex, and that every compact is eventually almost orthogonal to the sequence $(e^*_n)_n$.

Finally, for each $n\in\mathbb{N}$, we define $\alpha_n=(v_n,v_n^*,e_n,e_n^*,e_{n+1},e_{n+1}^*,t_n)$. It is straightforward to check that $\alpha_n$ satisfies the needed conditions in order to define the norm $\|\cdot\|_{\alpha_n}$ (i.e.: $t_n<\frac{\rho}{16}$ and equation \eqref{eq:biorthogonality_alpha}).

We now choose a suitable sequence $(\eta_n)_n$ to produce the approximations. Fix a sequence $(\varepsilon_i)_i$ of positive numbers such that $\varepsilon_i<2^{-i}\frac{\rho^2}{12800}$ for all $i\in\mathbb{N}$. We may choose $\varepsilon_i$ with $1+\varepsilon_i<\left(1-\frac{\rho}{8}\right)^{-1}$ for all $i\in\mathbb{N}$. Finally, define for each $n\in\mathbb{N}$ the positive number $\eta_n = \varepsilon_i$, where $i$ is the unique natural number such that $n\in A_i$. 

In summary, we have chosen, for every $n\in\mathbb{N}$, a tuple $\alpha_n=(v_n,v_n^*,e_n,e_n^*,e_{n+1},e_{n+1}^*,t_n)$ and a positive number $\eta_n$ such that:
\begin{enumerate}
    \item $(t_n)_n$  is a sequence of positive numbers satisfying $0<t_n<\frac{\rho}{16}$ for all $n\in\mathbb{N}$. Moreover, for every $t>0$ there exists $i\in\mathbb{N}$ such that $t_n<t$ for all $n\in A_i$.
    \item $(v_n)_n$ is a dense sequence in $S_{(X,\|\cdot\|)}$, with $v_n^*\in S_{(X,\|\cdot\|)}$ and $v_n^*(v_n)=1$ for all $n\in\mathbb{N}$. Moreover, the subsequence $\{v_n\colon n\in A_i\}$ is dense in $S_{(X,\|\cdot\|)}$ for all $i\in\mathbb{N}$.
    \item $(e_n)_n\subset S_{(X,\|\cdot\|)}$ and $(e^*_n)_n\subset S_{(X^*,\|\cdot\|)}$ form an ``almost biorthogonal" system, i.e.: $e^*_n(e_n)=1$ for every $n\in\mathbb{N}$, and $e^*_m(e_n)<\frac{\rho}{800}$ for all $n\neq m\in\mathbb{N}$.
    \item The sequence $(e^*_n)_n$ converges to $0$ in the weak$^*$ topology in $X^*$.
    \item $(\eta_n)_n$ is a sequence of positive numbers that satisfies $0<\eta_n<\min\left\{\frac{t_n\rho}{12800},\left(1-\frac{\rho}{8}\right)^{-1}-1\right\}$ for all $n\in\mathbb{N}$. Moreover, if $n\in A_i$ for $n,i\in\mathbb{N}$, then $\eta_n=\varepsilon_i$.
\end{enumerate}

Now, let $\|\cdot\|_n$ be a locally uniformly convex norm in $X$ such that
\begin{equation}
\label{eq:LUR_norms_app_alpha}
\|\cdot\|_{\alpha_n}\leq \|\cdot\|_n\leq(1+\eta_n)\|\cdot\|_{\alpha_n}.
\end{equation}
The final norm we consider on $X$ is defined as:
\begin{equation}
\label{eq:Definition_final_norm}
\vertiii{x}_\rho=\sup_{n\in\mathbb{N}}\left\{\|x\|_n,\left(1-\frac{\rho}{8}\right)^{-1}\|x\|\right\},\qquad\text{ for all }x\in X.
\end{equation}

We are now ready to prove the main result of the section:
\begin{theorem}
    \label{Th:Main_Theorem}
    The space $(X,\vertiii{\cdot}_\rho)$ is a locally uniformly convex space, $(1-\rho)^{-1}$-isomorphic to $(X,\|\cdot\|)$, and such that for any non-trivial convex compact subset is continuous but not uniformly continuous.
\end{theorem}
\begin{proof}

We first verify that $(X,\vertiii{\cdot}_\rho)$ is $(1-\rho)^{-1}$-isomorphic to $(X,\|\cdot\|)$. For a point $x\in S_{(X,\|\cdot\|)}$, it holds that $\vertiii{x}_\rho\geq \left(1-\frac{\rho}{8}\right)^{-1}$. On the other hand, by equations \eqref{eq:Eq_Alpha_norms} and \eqref{eq:LUR_norms_app_alpha}, we have $\|x\|_n\leq (1+\eta_n)(1-\rho)^{-1}\|x\|$ for every $n\in\mathbb{N}$. We conclude that
\begin{equation}
    \label{eq:Final_norm_eq}
    \left(1-\frac{\rho}{8}\right)^{-1}\|x\|\leq \vertiii{x}_\rho\leq \left(1-\frac{\rho}{8}\right)^{-1}(1-\rho)^{-1}\|x\|,\qquad\text{ for all }x\in X.
\end{equation}

Next we show that $(X,\vertiii{\cdot}_\rho)$ is locally uniformly convex. Fix $x\in S_{(X,\|\cdot\|)}$. Given $n\in\mathbb{N}$ such that $|e^*_n(x)|<1-\rho$, we have $\|x\|_{\alpha_n}=1$, which means that $\|x\|_n\leq (1+\eta_n)<\left(1-\frac{\rho}{8}\right)^{-1}$, and thus $\|x\|_n$ does not participate in the supremum which defines $\vertiii{x}_\rho$. Since the sequence $(e^*_n)_n$ converges to $0$ in the weak$^*$ topology of $X^*$, we deduce that there exists an open neighbourhood $U$ of $x$ and a natural number $n_x\in\mathbb{N}$ such that $\vertiii{z}_\rho=\max_{n\leq n_x}\left\{\|x\|_n,\left(1-\frac{\rho}{8}\right)^{-1}\|x\|\right\}$ for all $z\in U$. This shows that the norm $\vertiii{\cdot}_\rho$ is locally defined by finite intersection of locally uniformly convex norms, and is therefore locally uniformly convex itself. 

To show the final part of the theorem, we will apply Theorem \ref{Th:Estimate_modulus_NPM} for any non-trivial normalized convex compact set and for arbitrarily small $t>0$. Consider a convex compact set $K$ in $X$ with at least two points. By translating and dilating $K$ if necessary, we may assume without loss of generality that $K\subset B_{(X,\vertiii{\cdot}_\rho)}$, and that there exists $v_K\in S_{(X,\|\cdot\|)}$ such that $\left[-\frac{\rho}{4}v_K,\frac{\rho}{4}v_K\right]$ is contained in $K$. 

Consider an arbitrary $t>0$, and fix $i_0\in\mathbb{N}$ such that $t_n<t$ for all $n\in A_{i_0}$. Using that any cofinite subset of the sequence $\{v_n\colon n\in A_{i_0}\}$ is dense in $S_{(X,\|\cdot\|)}$, and that the sequence $(e^*_n)_n$ is weak$^*$ null, we obtain that there exists $n_0\in A_{i_0}$ such that $\frac{\rho}{4}\|v_{n_0}-v_K\|<\eta_{n_0}=\varepsilon_{i_0}$ and $\text{sup}_{k\in K}|e^*_{n_0}(k)|,\text{sup}_{k\in K}|e^*_{n_0+1}(k)|<\frac{t_{n_0}\rho}{12800}$. Note as well that $\eta_{n_0}<\frac{t_{n_0}\rho}{12800}$ by choice of $\varepsilon_{i_0}$. Moreover, equation \eqref{eq:Final_norm_eq} shows in particular that $\|\cdot\|_{\alpha_{n_0}}\leq \vertiii{\cdot}_\rho\leq (1-\rho)^{-2}\|\cdot\|$. 

To apply Theorem \ref{Th:Estimate_modulus_NPM}, it only remains to show that $\vertiii{y^\pm_{\alpha_{n_0}}}_\rho<1+\eta_{n_0}$. First, as computed in the previous section, $\|y^\pm_{\alpha_{n_0}}\|\leq 1-\frac{\rho}{8}$, while $\|y^\pm_{\alpha_{n_0}}\|_{\alpha_{n_0}}=1$. Moreover, for any $m\in\mathbb{N}$ different from ${n_0}$, it holds that 
\begin{align*}
e^*_m(y^\pm_{\alpha_{n_0}})&\leq\left(1-\frac{\rho}{2}\right)|e^*_m(e_{n_0})|+t_{n_0}|e^*_m(e_{{n_0}+1})|+\frac{\rho}{4}|e^*_m(v_{n_0})|\\
&\leq \left(1-\frac{\rho}{2}\right)\frac{\rho}{800}+\frac{\rho}{16}+\frac{\rho}{4}<1-\rho,
\end{align*}
since $\rho<\frac{1}{4}$. This implies that $\|y^\pm_{\alpha_{n_0}}\|_{\alpha_m}=\|y^\pm_{\alpha_{n_0}}\|\leq 1-\frac{\rho}{8}$. With these estimates, equation \eqref{eq:LUR_norms_app_alpha} and the definition of $\vertiii{\cdot}_\rho$ (equation \eqref{eq:Definition_final_norm}) we deduce that $\vertiii{y^\pm_{\alpha_{n_0}}}_\rho\leq 1+\eta_{n_0}$. 

We can apply now Theorem \ref{Th:Estimate_modulus_NPM}, and we obtain that the nearest point map $R\colon (X,\vertiii{\cdot}_\rho)\rightarrow (K,\vertiii{\cdot}_\rho)$ onto $K$ satisfies $\omega_{R}((1-\rho)^{-2}2t_{n_0})>\frac{\rho}{16}$. In particular, $\omega_{R}((1-\rho)^{-2}2t)>\frac{\rho}{16}$. Since this was done for arbitrary $t>0$, we conclude that $R$ is not uniformly continuous. Since $(X,\vertiii{\cdot}_\rho)$ is strictly convex, $R$ is continuous. 
\end{proof}

\begin{remark}
    In Theorem \ref{Th:Main_Theorem} the uniform continuity is spoiled by means of a pair of bounded sequences in $X$. Hence, we deduce the seemingly stronger result that the nearest point maps considered in Theorem \ref{Th:Main_Theorem} are not uniformly continuous even when restricted to the unit ball.
\end{remark}

We can now restate and prove the main result of the article as a direct corollary:

\begin{corollary}[Theorem \ref{THEOREM_A}]
    Every infinite-dimensional separable Banach space is $(1+\varepsilon)$-isomorphic to a locally uniformly convex Banach space in which the nearest point map onto every non-trivial convex and compact set is continuous but not uniformly continuous, for any $\varepsilon>0$.
\end{corollary}
\begin{proof}
    Note that every separable Banach space is $(1+\varepsilon)$-isomorphic to a locally uniformly convex Banach space, for any $\varepsilon>0$. Therefore, using Theorem \ref{Th:Main_Theorem} and equation \eqref{eq:Final_norm_eq}, it suffices to apply the construction of the norm $\vertiii{\cdot}_\rho$ described above in any infinite-dimensional separable locally uniformly convex Banach space for small enough $0<\rho<1/4$.
\end{proof}
\section*{Acknowledgements}
This work was supported by MCIN/AEI/10.13039/501100011033/FEDER, UE: grants PID2021-122126NB-C31 (R. Medina) and PID2021-122126NB-C33 (A. Quilis)

The research of R. Medina was also supported by FPU19/04085 MIU (Spain) Grant, by Junta de Andalucia Grant FQM-0185 by GA23-04776S project (Czech Republic) and by SGS22/053/OHK3/1T/13 project (Czech Republic).

The research of A. Quilis was also supported by the French ANR project No. ANR-20-CE40-0006.
\printbibliography
\end{document}